\documentclass[12pt,reqno]{amsart}
\usepackage{amssymb,latexsym,amsmath,amscd, graphicx}
\usepackage[colorlinks=true,pagebackref,hyperindex]{hyperref}
\usepackage{verbatim}
\usepackage{xypic}
\usepackage{tikz}

\usepackage{amsfonts}
\usepackage{amscd}

\renewcommand{\phi}{\varphi}
\renewcommand{\epsilon}{\varepsilon}
\renewcommand{\theta}{\vartheta}

\def\ZZ{{\mathbf Z}}

\def\RR{{\mathbf R}}
\def\QQ{{\mathbf Q}}

\def\cI{\mathcal{I}}
\def\cJ{\mathcal{J}}

\def\cF{\mathcal{F}}

\def\cO{\mathcal{O}}
\def\cE{\mathcal{E}}

\def\fra{\mathfrak{a}}

\DeclareMathOperator{\ord}{ord}

\DeclareMathOperator{\NS}{N}
\DeclareMathOperator{\SB}{{\mathbf S}{\mathbf B}}
\DeclareMathOperator{\B}{{\mathbf B}}

\newtheorem{lemma}{Lemma}[section]
\newtheorem{theorem}[lemma]{Theorem}
\newtheorem{corollary}[lemma]{Corollary}
\newtheorem{proposition}[lemma]{Proposition}

\theoremstyle{definition}

\newtheorem{remark}[lemma]{Remark}

\theoremstyle{remark}
\newtheorem*{remark*}{Remark}
\newtheorem*{note*}{Note}

\numberwithin{equation}{section}

\begin{document}

\title[NUMERICAL DIMENSION IN POSITIVE CHARACTERISTIC]{On the numerical dimension of pseudo-effective divisors in positive characteristic}

\author[P.~Cascini]{Paolo~Cascini}
\address{Department of Mathematics, Imperial College London, London SW7 2AZ, UK}
\email{{p.cascini@imperial.ac.uk}}

\author[C.~Hacon]{Christopher~Hacon}
\address{Department of Mathematics, University of Utah, Salt Lake City, UT 84112, USA}
\email{hacon@math.utah.edu}

\author[M.~Musta\c{t}\u{a}]{Mircea~Musta\c{t}\u{a}}
\address{Department of Mathematics, University of Michigan,
Ann Arbor, MI 48109, USA}
\email{mmustata@umich.edu}

\author[K.~Schwede]{Karl~Schwede}
\address{Department of Mathematics, The Pennsylvania State University,
University Park, PA 16802, USA}
\email{schwede@math.psu.edu}

%\thanks{2010\,\emph{Mathematics Subject Classification}.
\thanks{Cascini was partially supported by EPSRC grant P28327, Hacon was partially supported by NSF grant DMS-0757897 and a grant from the Simons Foundation, Musta\c{t}\u{a}
was partially supported by
 NSF grant DMS-1068190, and Schwede was  partially supported by NSF grant DMS-1064485 and a Sloan Foundation Fellowship.}

\begin{abstract}
Let $X$ be a smooth projective variety over an algebraically closed field of positive characteristic. We prove that if $D$ is a pseudo-effective $\RR$-divisor on $X$ which is not
numerically equivalent to the negative part in its divisorial Zariski decomposition, then the
numerical dimension of $D$ is positive. In characteristic zero, this was proved by Nakayama
using vanishing theorems.
\end{abstract}

\maketitle

%\markboth{P.~CASCINI, M.~MUSTA\c{T}\u{A}, AND K.~SCHWEDE}
%{NUMERICAL DIMENSION IN POSITIVE CHARACTERISTIC}

\section{Introduction}

Let $X$ be a smooth projective variety over an algebraically closed field $k$.
Recall that an $\RR$-divisor on $X$ is \emph{pseudo-effective} if its numerical class lies in the
closure of the set of classes of effective $\RR$-divisors on $X$. Following \cite{Nakayama},
we define the \emph{numerical dimension} $\kappa_{\sigma}(D)$ of an $\RR$-divisor $D$ as
the largest non-negative integer $\ell$ such that for some ample divisor $A$ one has
$$\liminf_{m\to\infty} \frac {h^0(X,\cO_X(\lfloor mD\rfloor+A))}{m^{\ell}}>0$$
(if there is no such $\ell\geq 0$, then $\kappa_{\sigma}(D)=-\infty$). In fact, $\kappa_{\sigma}(D)
\geq 0$ if and only if $D$ is pseudo-effective.
The following is the main result of this note:

\begin{theorem}\label{thm_main}
Suppose that $X$ is a smooth projective variety over an algebraically closed field $k$ of positive characteristic. If $D$ is a pseudo-effective $\RR$-divisor
on $X$ which is not numerically equivalent to the negative part $N_{\sigma}(D)$ in its divisorial Zariski
decomposition, then $\kappa_{\sigma}(D)\geq 1$, that is, there is an ample divisor
$A$ on $X$ and $C>0$ such that
$$h^0(X,\cO_X(\lfloor mD\rfloor+A))\geq Cm\,\,\text{for all}\,\,m\gg 0.$$
\end{theorem}

The above theorem was proved by N.~Nakayama in \cite[Thm.~V.1.12]{Nakayama} in characteristic
zero, using the Kawamata-Viehweg vanishing theorem. That result turned out to be  useful in many situations
when dealing with pseudo-effective divisors. For example, it was used in the proof of the
non-vanishing theorem, an important ingredient in proving the finite generation of the canonical ring, see \cite[Lem.~6.1]{BCHM}. We note that a large part of the results in birational geometry in characteristic zero
rely on the use of vanishing theorems. The above result illustrates
our belief that in spite of the failure of vanishing theorems in positive characteristic,
several results can still be recovered by making systematic use of the Frobenius morphism.

Let us say a few words about the divisorial Zariski decomposition that appears in the above theorem; for details, see \S 2. If $D$ is a pseudo-effective divisor, then to every prime divisor
$\Gamma$ on $X$ one can associate
a non-negative real number $\sigma_{\Gamma}(D)$ that only depends on the numerical class of $D$. It is known that there are only
finitely many prime divisors $\Gamma$ with $\sigma_{\Gamma}(D)>0$.
The \emph{negative part
in the divisorial Zariski decomposition of $D$} is
$$N_{\sigma}(D):=\sum_{\Gamma}\sigma_{\Gamma}(D)\Gamma.$$

Given an arbitrary $\RR$-divisor $D$ on $X$, the
\emph{non-nef locus} of $D$ is the union
$${\mathbf B}_-(D):=\bigcup_A\SB(D+A),$$
where $A$ varies over the ample divisors $A$ on $X$ such that $D+A$ is a $\QQ$-divisor,
and $\SB(E)$ denotes the stable base locus of a $\QQ$-divisor $E$.
One can show that, in fact, ${\mathbf B}_-(D)$ is a countable union of Zariski closed subsets.
 Note that
${\mathbf B}_-(D)$ is empty if and only if $D$ is nef.
If we assume that
the ground field is uncountable, then ${\mathbf B}_-(D)=X$
if and only if $D$ is not pseudo-effective. Furthermore, in this case
a prime divisor $\Gamma$ is contained in ${\mathbf B}_-(D)$ if and only if
$\sigma_{\Gamma}(D)>0$.

The assertion in Theorem~\ref{thm_main} is proved by showing that after
replacing $D$ by $D-N_{\sigma}(D)$, we can find an ample divisor $A$ such that the sections
of $\cO_X(\lfloor mD\rfloor+A)$ can be lifted from suitable curves to $X$.
The following is another result in this direction.

\begin{theorem}\label{p_lifting}
Let $X$ be a smooth projective variety over an algebraically closed field of positive characteristic, and let $D$ be an $\RR$-divisor on $X$. Suppose that $H=H_1+\ldots+H_r$ is a simple normal crossing divisor on
$X$, with $r<\dim(X)$,
 and that
$W:=H_1\cap\dots\cap H_r$  does not intersect the non-nef locus $\mathbf B_-(D)$ of $D$.
In this case there exists an ample divisor $A$ on $X$ such that the restriction map
\begin{equation*}\label{0}
H^0(X,\cO_X(\lfloor mD\rfloor +A))\to H^0(W,\cO_X(\lfloor mD\rfloor +A)\vert_W)
\end{equation*}
is surjective for every $m\ge 1$.
\end{theorem}

For a version of this result in characteristic zero, see \cite[Prop.~V.1.14]{Nakayama}. The proofs
of Theorems~\ref{thm_main} and \ref{p_lifting} make essential use of the Frobenius morphism.

In any characteristic, there are several possible candidates for the notion of numerical dimension
of a pseudo-effective divisor. In characteristic zero, B.~Lehmann showed in
\cite{Lehmann} that all these definitions are equivalent. We expect that a similar result
holds also in positive characteristic, but we do not pursue this direction here.
We only show, using Theorem~\ref{p_lifting}, that in the case of a nef divisor $D$, the
numerical dimension of $D$ as defined above can also be described as the largest
integer $j\geq 0$ such that the cycle class $D^j$ is not numerically trivial (this is the
definition in \cite{Kawamata}).  This is done in Proposition \ref{prop.KawamataNumericalDimensionEquivalent}.

The paper is organized as follows. In the next section we collect some basic facts about
pseudo-effective divisors, non-nef loci, and divisorial Zariski decompositions.
\S 3 is devoted to proving a general vanishing result, valid in arbitrary characteristic, which
only uses
asymptotic Serre vanishing. This is then used in \S 4
to prove the two results stated above.

\subsection*{Acknowledgment}
We are grateful to J\'{a}nos Koll\'{a}r and Burt Totaro for comments on a preliminary version of the manuscript, and to
Rob Lazarsfeld for some helpful discussions.

\section{Review of basic invariants of pseudo-effective divisors}

We start by recalling some basic definitions and notation. We work over a fixed
algebraically closed field $k$. For now, we do not make any assumption on the characteristic of $k$. All varieties are assumed to be reduced and irreducible.
Let $X$ be an $n$-dimensional normal projective variety over $k$.
A divisor
($\QQ$-divisor, $\RR$-divisor) on $X$ is a linear combination with integer (respectively, rational
or real) coefficients of prime divisors. If $D=\sum_{i=1}^da_iD_i$ is an $\RR$-divisor on $X$
(with the $D_i$ distinct prime divisors), then we put
$\lfloor D\rfloor:=\sum_{i=1}^d\lfloor a_i\rfloor D_i$ and $\lceil D\rceil:=\sum_{i=1}^d\lceil a_i
\rceil D_i$, where for a real number $u$, we denote by $\lfloor u\rfloor$ and $\lceil u\rceil$
the largest (respectively, smallest) integer that is $\leq u$ (respectively, $\geq u$).
We denote by ${\rm Cart}(X)$ the group of Cartier divisors on $X$.
An $\RR$-Cartier $\RR$-divisor on $X$ is an element of ${\rm Cart}(X)\otimes_{\ZZ}\RR$.
We consider this real vector space as a subspace of the space of $\RR$-divisors.
We will mostly be concerned with the case when $X$ is smooth, when every $\RR$-divisor
is $\RR$-Cartier.

We denote by
$\NS^1(X)_{\RR}$ the quotient of ${\rm Pic}(X)\otimes_{\ZZ}\RR$ by the numerical equivalence
relation. This is a finite-dimensional real vector space.
 The cone of \emph{pseudo-effective} classes is the closure
of the cone generated by classes of line bundles $L$ with $h^0(X,L)\geq 1$.
Its interior is the cone of \emph{big} classes.

Suppose now that $X$ is a smooth projective variety.
If $D$ is a divisor, then we denote by ${\rm Bs}(|D|)$ the base-locus of the linear system
$|D|$, considered with the reduced scheme structure. If $D$ is a $\QQ$-divisor, then
the \emph{stable base locus} $\SB(D)$ is the intersection
$\bigcap_m{\rm Bs}(|mD|)$, where $m$ varies over the positive integers such that
$mD$ has integer coefficients. This is a Zariski closed subset, and it is a consequence
of the Noetherian property that $\SB(D)={\rm Bs}(|mD|)$ whenever $m$ is divisible enough.

The \emph{non-nef locus} of an $\RR$-divisor $D$ is defined as
$$\B_-(D):=\bigcup_A\SB(D+A),$$ where $A$ varies over the
ample $\RR$-divisors such that $D+A$ has $\QQ$-coefficients. It is easy to see
that if $(A_m)_{m\geq 1}$ is a sequence of ample $\RR$-divisors whose classes in $\NS^1(X)_{\RR}$ go to zero,
and such that $D+A_m$ is a $\QQ$-divisor for every $m$,
then $\B_-(D)=\bigcup_{m\geq 1}\SB(D+A_m)$. In particular,
$\B_-(D)$ is a countable union of Zariski closed subsets (it was recently shown by John Lesieutre that $\B_-(D)$ is not always closed \cite{Lesieutre}).

It follows from definition that $\B_-(D)$ only depends on the numerical equivalence
class of $D$. Furthermore,
$\B_-(D)$ is empty if and only if $D$ is nef; if $k$ is uncountable, then
$\B_-(D)$ is a proper subset of $X$ if and only if $D$ is pseudo-effective.
For these facts about the non-nef locus and for a more detailed discussion, see
\cite{ELMNP}.

We now recall the asymptotic order of vanishing function and the divisorial Zariski decomposition, both due to Nakayama \cite{Nakayama}. For the proofs of the results
that we list below we refer to \cite[Chap. III]{Nakayama} for the case ${\rm char}(k)=0$,
and to \cite{Mustata} for the case ${\rm char}(k)>0$. Let $\Gamma$ be a prime divisor on $X$.
Suppose first that $D$ is a big $\QQ$-divisor. Consider a positive integer $m$ such that $mD$ has integer coefficients and \mbox{$h^0(X,\cO_X(mD))>0$.} In this case we denote by
$\ord_{\Gamma}|mD|$ the coefficient of $\Gamma$ in a general element in $|mD|$.
One defines
$$\ord_{\Gamma}\parallel D\parallel:=\inf_m\frac{\ord_{\Gamma}|mD|}{m}=
\lim_{m\to\infty}\frac{\ord_{\Gamma}|mD|}{m},$$
where $m$ is as above.
One can show that $\ord_{\Gamma}\parallel D\parallel$ only depends on the numerical
equivalence class of $D$, and the map $D\to \ord_{\Gamma}\parallel D\parallel$ extends to a continuous
map (denoted in the same way) on the big cone. Furthermore, if $D$ is a pseudo-effective
$\RR$-divisor on $X$, we put
\begin{equation}\label{def_order}
\sigma_{\Gamma}(D):=
\sup_A\ord_{\Gamma}\parallel D+A\parallel,
\end{equation}
where $A$ varies over the ample $\RR$-divisors on $X$, or equivalently,
over a sequence of ample $\RR$-divisors whose classes in $\NS^1(X)_{\RR}$ go to zero.
It follows from definition that if $D'$ is a big $\RR$-divisor and $A$ is ample, then
$$\ord_{\Gamma}\parallel D'\parallel (\Gamma\cdot A^{n-1})\leq (D'\cdot A^{n-1}).$$
We deduce that if $D$ is pseudo-effective and $A$ is a fixed ample $\RR$-divisor, then
$$\sigma_{\Gamma}(D)\leq\lim_{\epsilon\to 0}\frac{((D+\epsilon A)\cdot A^{n-1})}
{(\Gamma\cdot A^{n-1})}=\frac{(D\cdot A^{n-1})}{(\Gamma\cdot A^{n-1})}<\infty.$$
The function
$\sigma_{\Gamma}$ is lower semi-continuous on the pseudo-effective cone,
and it agrees with $\ord_{\Gamma}\parallel-\parallel$ on the big cone.

\begin{proposition}\label{char_non_nef}
Let $D$ be a pseudo-effective $\RR$-divisor on the smooth projective variety $X$, and let
$\Gamma$ be a prime divisor on $X$.
\begin{enumerate}
\item[i)] If $(A_m)_{m\geq 1}$ is a sequence of ample $\RR$-divisors whose classes
in $\NS^1(X)_{\RR}$ go to zero, and such that all
$D+A_m$ are $\QQ$-divisors, then $\sigma_{\Gamma}(D)=0$ if and only if
$\Gamma\not\subseteq \SB(D+A_m)$ for every $m\geq 1$.
\item[ii)] If $k$ is uncountable, then $\sigma_{\Gamma}(D)=0$ if and only if
$\Gamma\not\subseteq\B_-(D)$.
\end{enumerate}
\end{proposition}

\begin{proof}
For the assertion in ii), see \cite[Prop. 2.8]{ELMNP} and \cite[Thm. 7.2]{Mustata}
for the cases when the ground field has characteristic zero or positive, respectively.
Note that in order to check the assertion in i), we may extend the ground field and therefore assume that it is uncountable. In this case,
i) is just a reformulation of ii).
\end{proof}

It was shown in \cite[Cor. III.1.10]{Nakayama} that if $D$ is a pseudo-effective $\RR$-divisor on $X$ and
$\Gamma_1,\ldots,\Gamma_r$ are mutually
distinct prime divisors with $\sigma_{\Gamma_i}(D)>0$ for all $i$, then the classes
of $\Gamma_1,\ldots,\Gamma_r$ in $\NS^1(X)_{\RR}$ are linearly independent
(the proof therein is characteristic-free).
In particular, $r$ is bounded above by $\dim_{\RR}\NS^1(X)_{\RR}$, hence there are only
finitely many $\Gamma$ with $\sigma_{\Gamma}(D)>0$. One defines
$$N_{\sigma}(D):=\sum_{\Gamma}\sigma_{\Gamma}(D)\Gamma,\,\,\,\,P_{\sigma}(D):=
D-N_{\sigma}(D).$$
The decomposition $D=N_{\sigma}(D)+P_{\sigma}(D)$ is known as the
\emph{divisorial Zariski decomposition} of $D$, and $N_{\sigma}(D)$ and $P_{\sigma}(D)$
are the negative and the positive part, respectively, of this decomposition. Note that
$N_{\sigma}(D)$ is an effective $\RR$-divisor, and it only depends on the numerical
class of $D$.

\begin{proposition}\label{prop_Nakayama}
Let $D$ be a pseudo-effective $\RR$-divisor as above, and $D=N_{\sigma}(D)+P_{\sigma}(D)$
its divisorial Zariski decomposition. For every $\RR$-divisor $F$ with $0\leq F\leq
N_{\sigma}(D)$, the divisor $D-F$ is pseudo-effective and $N_{\sigma}(D-F)=N_{\sigma}(D)-F$.
In particular, if $A$ is an ample $\RR$-divisor such that $P_{\sigma}(D)+A$ is a $\QQ$-divisor,
then $\SB(P_{\sigma}(D)+A)$ contains no subvarieties of codimension one.
\end{proposition}

\begin{proof}
The first assertion is proved in \cite[Lem.~III.1.8]{Nakayama}, and the proof therein is independent of characteristic. This implies  that for every prime divisor $\Gamma$, we have
$\sigma_{\Gamma}(P_{\sigma}(D))=0$. The second assertion now follows from Proposition~\ref{char_non_nef}.
\end{proof}

For future reference we include the following two lemmas. Both results are well-known,
but we include the proofs for the benefit of the reader.

\begin{lemma}\label{lem_num_trivial1}
If $D$ is a pseudo-effective $\RR$-divisor on the smooth $n$-dimensional projective variety $X$
and $D$ is not numerically trivial, then for every ample $\RR$-divisors
$A_1,\ldots,A_{n-1}$ on $X$, we have $(D\cdot A_1\cdot\ldots\cdot A_{n-1} )>0$.
\end{lemma}

\begin{proof}
%Without loss of generality we may assume that the base field is uncountable.
Since $D$ is pseudo-effective, it is clear that $(D\cdot A_1\cdot\ldots\cdot  A_{n-1} )\geq 0$,
so we only need to show that this intersection number is nonzero.
This is clear if $X$ is a curve, hence from now on we assume that $n\geq 2$.
Note also that the intersection number we are interested in does not change if we extend the ground field, hence we may assume that the ground field is uncountable.

 If $A'_1,\ldots,A'_{n-1}$ are ample $\QQ$-divisors such that each $A_i-A'_i$ is ample, then
$$(D\cdot A_1\cdot\ldots\cdot A_{n-1})\geq (D\cdot A'_1\cdot\ldots\cdot A'_{n-1}).$$ Therefore we may assume that
each $A_i$ is a $\QQ$-divisor. Furthermore, we may replace each $A_i$ by a multiple and so we may assume that $A_i$ has integer coefficients.

Since $D$ is not numerically trivial, there is an irreducible curve $C$ in $X$ such that
$(D\cdot C)\neq 0$. If $n\geq 3$, let
 $\pi\colon Y\to X$ be the blow-up of $X$ along $C$, with exceptional divisor
$E$.  For $m\gg 0$, the line bundle $\cO_Y(m\pi^*(A_{n-1})-E)$ is very ample. Bertini's theorem
then implies that if $H$ is a general element of $|mA_{n-1}|$ that vanishes along $C$, then
$H$ is irreducible. If $H=r\widetilde{H}$, where $\widetilde{H}=H_{\rm red}$, then
$$(D\cdot A_1\ldots \cdot A_{n-1})=\frac{1}{m}(D\vert_H\cdot A_1\vert_H\cdot\ldots\cdot
 A_{n-2}\vert_H)=
\frac{r}{m}
(D\vert_{\widetilde{H}}\cdot A_1\vert_{\widetilde{H}}\cdot\ldots\cdot
 A_{n-2}\vert_{\widetilde{H}}).$$
Furthermore, if the class of $D$ is the limit of effective $\QQ$-divisor classes $D_m$, then by taking
$H$ to be very general (meaning, chosen outside a countable union of proper closed subvarieties of $|mA_{n-1}|$), we may assume that each $D_m\vert_{\widetilde{H}}$ is effective,
hence $D\vert_{\widetilde{H}}$ is pseudo-effective. Note also that since
$C\subset \widetilde{H}$, it is clear that $D\vert_{\widetilde{H}}$ is not numerically trivial.

After iterating this argument, we find an irreducible and reduced surface $S\subseteq X$
containing $C$, such that $D\vert_S$ is a pseudo-effective $\RR$-divisor and such that
$(D\cdot A_1\cdot\ldots\cdot A_{n-1} )$ is a positive multiple of
$(D\vert_S\cdot A_1\vert_S)$.  If this is zero,
then the linear map
$D'\to (D\vert_S\cdot D')$ vanishes at a point in the interior of the nef cone of $S$.
On the other hand, it is non-negative on the nef cone of $S$, since $D\vert_S$
is pseudo-effective. Therefore the map is identically zero. In order to obtain a contradiction, it is enough
to show that $(D\vert_S^2)\neq 0$.

Let $f\colon S'\to S$ be a resolution of singularities of $S$ (since $S$ is a surface, such a resolution is known to exist in arbitrary characteristic). Since $(f^*(A_1\vert_S))^2=
(A_1\vert_S^2)>0$ and $(f^*(D\vert_S)\cdot f^*(A_1\vert_S))=(D\vert_S\cdot A_1\vert_S)=0$,
it follows from the Hodge index theorem that $(f^*(D\vert_S)^2)\leq 0$, with equality if and only if
$f^*(D\vert_S)$ is numerically trivial. However, if $C'\subset S'$ is an irreducible curve that
dominates $C$, then $(f^*(D\vert_S)\cdot C')\neq 0$. Therefore
$(D\vert_S^2)<0$, which completes the proof of the lemma.
\end{proof}

\begin{lemma}\label{lem_num_trivial2}
If $D$ is a nef $\RR$-divisor on the smooth $n$-dimensional projective variety $X$
and  $0\leq j\leq n$ is such that the cycle class
$D^j$ is not numerically trivial, then for all ample $\RR$-divisors
$A_1,\ldots,A_{n-j}$ on $X$ we have
$(D^j\cdot A_1\cdot\ldots\cdot A_{n-j})>0$.
\end{lemma}

\begin{proof}
By definition, the fact that $D^j$ is not numerically trivial means that there is a polynomial $P$ in Chern classes
of vector bundles on $X$ such that ${\rm deg}(P\cap D^j)\neq 0$. However, since $X$ is nonsingular, it follows from the Grothendieck-Riemann-Roch  theorem that, in fact, we can find
an irreducible subvariety $W$ of $X$ of dimension $j$ such that $(D^j\cdot W)\neq 0$
(see \cite[Example~19.1.5]{Fulton}).

We may assume that $1\leq j\leq n-1$, as otherwise the assertion in the lemma is trivial.
Arguing as in the proof of Lemma~\ref{lem_num_trivial1}, we may assume that all $A_i$ have integer coefficients. Furthermore, we
can find an irreducible and reduced subvariety $Y$ of $X$ of dimension $j+1$ such that $W\subseteq Y$ and $(D^j\cdot A_1\cdot\ldots\cdot A_{n-j})$ is a positive multiple of
$(D\vert_Y^j\cdot A_1\vert_Y)$. If $m\gg 0$, we can find  $Z$ in
$|mA_1\vert_Y|$ such that $W$ is an irreducible component of $Z$. Using the fact that $D$ is nef,
we conclude that
$$(D\vert_Y^j\cdot A_1\vert_Y)=\frac{1}{m}(D^j\vert_Z)\geq\frac{1}{m}(D\vert_W^j)>0.$$
This completes the proof of the lemma.
\end{proof}

We now recall the definition of numerical dimension. Let $X$ be a smooth, projective variety
as above, and $D$ an $\RR$-divisor on $X$. The \emph{numerical dimension}
of $D$ is defined by
$$\kappa_\sigma(D):=\max\{\ell\in \ZZ_{\geq 0}\mid \liminf_{m\to\infty} \frac {h^0(X,\cO_X(\lfloor mD\rfloor+A))}{m^{\ell}}>0\quad \text{for some ample divisor }A \}$$
(by convention, if the above set is empty, then $\kappa_{\sigma}(D)=-\infty$).

Note that if $A$ satisfies the condition in the definition of $\kappa_{\sigma}(D)$ for
$\ell$, then the same holds for all divisors
$A'$ such that $A'-A$ is effective. We also note that in the above definition we could replace the
round-down by the round-up function. More precisely, we have
$$\kappa_\sigma(D)=\max\{\ell\in \ZZ_{\geq 0}\mid \liminf_{m\to\infty} \frac {h^0(X,\cO_X(\lceil mD\rceil+A))}{m^{\ell}}>0\quad \text{for some ample divisor }A \}.$$
Indeed, let $T$ be the reduced effective divisor with the same support as $D$, and let $H$ be an ample divisor such that $H-T$ is effective. For every $m\geq 1$, the difference
$T-(\lceil mD\rceil-\lfloor mD\rfloor)$ is effective, hence
$H-(\lceil mD\rceil-\lfloor mD\rfloor)$ is effective. Therefore for every ample divisor $A$ we have
$$h^0(X,\cO_X(\lceil mD\rceil + A))\le h^0(X,\cO_X(\lfloor mD\rfloor + A+H)).$$
This proves our assertion.

We will see in Proposition~\ref{char_pseudoeffective} below that
$D$ is pseudo-effective if and only if $\kappa_{\sigma}(D)\geq 0$.
It is also easy to see that $\kappa_{\sigma}(D)=\dim(X)$ if and only if $D$ is big.

\begin{remark}
In \cite{Nakayama}, variants of the above definition were introduced.  For example,
the limit inferior was replaced by limit superior, or the condition
$\liminf>0$ was replaced by $\limsup<\infty$. It was shown in \cite{Lehmann} that in characteristic zero
all these variants give the same invariant, which also admits other characterizations in terms of volumes or positive intersection products. We expect that such a result should also hold in
positive characteristic, but we do not pursue this here.
\end{remark}

In this paper we are mainly concerned with smooth varieties. However, we now briefly discuss the
numerical dimension of $\RR$-Cartier $\RR$-divisors on normal varieties. In particular, we show that it can be computed as
the numerical dimension of the pull-back to any smooth alteration.

\begin{remark}\label{non_smooth1}
If $D$ is an $\RR$-Cartier $\RR$-divisor on a normal projective variety $X$, then we can still define the numerical dimension $\kappa_{\sigma}(D)$ by the same formula as above.
In this case, however, it is convenient to also use an alternative description, as follows. Consider
a sequence of Cartier divisors $(D_m)_{m\geq 1}$ on $X$, with the following property:
there are finitely many Cartier divisors $T_1,\ldots,T_r$ and $M>0$
such that we can write $D=\sum_{i=1}^rq_iT_i$ and $D_m=\sum_{i=1}^rq_{m,i}T_i$,
with $q_i\in\RR$ and $q_{m,i}\in\ZZ$
such that $|mq_i-q_{m,i}|\leq M$ for every $i$ and $m$.
We can always find such a sequence: there is such an expression for $D$ since $D$
is $\RR$-Cartier, and we can take $D_m=\sum_{i=1}^rq_{m,i}T_i$,  with
$q_{m,i}=\lfloor mq_i\rfloor$ or $q_{m,i}=\lceil mq_i\rceil$.
 Given \emph{any} sequence
$(D_m)_{m\geq 1}$ as above,
$$\kappa_{\sigma}(D)=\max\{\ell\in \ZZ_{\geq 0}\mid \liminf_{m\to\infty} \frac {h^0(X,\cO_X(D_m+A))}{m^{\ell}}>0\, \text{for some ample Cartier divisor }A \}.$$
The equality is a consequence of the fact that under our assumptions on $(D_m)_{m\geq 1}$,
all divisors $D_m-\lfloor mD\rfloor$ are supported on a finite set of prime divisors, and their coefficients are bounded. Therefore
we can find an ample Cartier divisor $H$ such that we have  $h^0(X,\cO_X(H+D_m-\lfloor mD\rfloor))\geq 1$ and $h^0(X,\cO_X(H-D_m+\lfloor mD\rfloor))\geq 1$ for all $m\geq 1$.
\end{remark}

Recall that an alteration $\pi \colon Y \to X$ is a  projective, surjective, and generically finite morphism.  The existence of alterations with smooth total space is guaranteed by a theorem of de Jong \cite{deJong}.

\begin{proposition}\label{alteration}
Let $X$ be a normal projective variety and $D$ an $\RR$-Cartier $\RR$-divisor on $X$. If
$\pi\colon Y\to X$ is an alteration, with $Y$ normal, then
$$\kappa_{\sigma}(D)=\kappa_{\sigma}(\pi^*(D)).$$
\end{proposition}

\begin{proof}
Let us write $D=\sum_{i=1}^rq_iT_i$, where each $T_i$ is a Cartier divisor and $q_i\in\RR$,
and put $D_m=\sum_{i=1}^r\lfloor mq_i\rfloor T_i$ for $m\geq 1$. It follows from
Remark~\ref{non_smooth1} that we can use the sequences $(D_m)_{m\geq 1}$
and $(\pi^*(D_m))_{m\geq 1}$ to compute $\kappa_{\sigma}(D)$ and
$\kappa_{\sigma}(\pi^*(D))$, respectively.

If $A$ is an ample Cartier divisor on $X$ and $B$ is a Cartier divisor on $Y$ such that
$B-\pi^*(A)$ is ample, then we have
$$h^0(Y,\cO_Y(\pi^*(D_m)+B))\geq h^0(X,\cO_X(D_m+A))\,\text{for every}\,m\geq 1,$$
hence $\kappa_{\sigma}(\pi^*(D))\geq\kappa_{\sigma}(D)$. In order to prove the opposite inequality, suppose that $H_Y$ is an ample divisor on $Y$ such that
$h^0(Y, \cO_Y(\pi^*(D_m)+H_Y))\geq Cm^{\ell}$ for some $C>0$ and all $m\gg 0$.
If $H$ is any ample Cartier divisor on $X$, then $\pi^*(H)$ is big on $Y$, hence we can find
$d>0$ such that $h^0(Y,\cO_Y(d\pi^*(H)-H_Y))\geq 1$, hence
$$h^0(X,\cO_X(D_m+dH)\otimes\pi_*(\cO_Y))=h^0(Y, \cO_Y(\pi^*(D_m+dH)))\geq Cm^{\ell}\,\,\text{for all}\,\,m\gg 0.$$

Note now that for a suitable $d'>0$, we can find an embedding
$\pi_*(\cO_Y)\hookrightarrow \cO_X(d'H)^{\oplus N}$ for a positive integer $N$. Indeed,
let $d'\gg 0$ be such that $\pi_*(\cO_Y)^{\vee}\otimes\cO_X(d'H)$ is generated by global sections (for a sheaf ${\mathcal F}$, we denote by ${\mathcal F}^{\vee}$ its dual
${\mathcal Hom}_{\cO_X}({\mathcal F}, {\mathcal O}_X)$). We thus have a surjection
$\cO_X(-d'H)^{\oplus N}\to \pi_*(\cO_Y)^{\vee}$, which induces an embedding
$$\pi_*(\cO_Y)\hookrightarrow(\pi_*(\cO_Y)^{\vee})^{\vee}\hookrightarrow\cO_X(d'H)^{\oplus N}.$$ We deduce that for $m\gg 0$, we have
$h^0(X,\cO_X(D_m+(d+d')H))\geq \frac{C}{N}m^{\ell}$, hence
$\kappa_{\sigma}(D)\geq \ell$. This completes the proof of the proposition.
\end{proof}

We conclude this section with the following characterization of pseudo-effective divisors.

\begin{proposition}\label{char_pseudoeffective}
If $X$ is a smooth projective variety, then there is a divisor $G$
on $X$ such that an $\RR$-divisor $D$ on $X$ is pseudo-effective if and only if
$h^0(X,\cO_X(\lceil mD\rceil+G))>0$ for every $m\geq 1$. In particular, $D$
is pseudo-effective if and only if $\kappa_{\sigma}(D)\geq 0$.
\end{proposition}

\begin{proof}
The assertion is well-known when the ground field has characteristic zero (see
\cite[Cor.~V.1.4]{Nakayama}), so we only give the argument when the ground field has positive characteristic. Let $H$ be a very ample divisor. We show that
$G=K_X+(n+2)H$ has the required property, where $K_X$ is such that $\cO_X(K_X)\simeq\omega_X$ and
 $n=\dim(X)$.

Suppose first that $D$ is pseudo-effective.
Note that $\lceil mD\rceil -mD$ is effective, hence $\lceil mD\rceil$ is pseudo-effective,
and therefore $E=\lceil mD\rceil +H$ is a big divisor. It now follows from
\cite[Thm.~4.1]{Mustata} that
$$\tau(\parallel E\parallel)\otimes_{\cO_X}\cO_X(K_X+E+(n+1)H)$$
is globally generated, where $\tau(\parallel E\parallel)$ is the asymptotic test ideal of $E$.
The important thing for us is that $\tau(\parallel E\parallel)$ is nonzero, which implies
$h^0(X,\cO_X(\lceil mD\rceil+G))>0$. Note that one can give a similar argument in characteristic zero, by replacing the asymptotic test ideal by the asymptotic multiplier ideal, and using
the corresponding global generation result (see \cite[Cor.~11.2.13]{positivity}).

The converse is clear (in fact, any $G$ satisfies this direction). Indeed, if we have
$h^0(X,\cO_X(\lceil mD\rceil+G))>0$ for every $m\geq 1$, then each
$\frac{1}{m}\lceil mD\rceil +\frac{1}{m}G$ is pseudo-effective. Therefore the limit $D$
of these divisors is pseudo-effective as well.

The last assertion in the proposition is now a consequence of the first one, and of the second description of
$\kappa_{\sigma}(D)$ (the one using $\lceil mD\rceil$ instead of $\lfloor mD\rfloor$).
\end{proof}

\begin{remark}\label{non_smooth2}
The characterization of pseudo-effective divisors in terms of  their numerical dimension,
given in Proposition~\ref{char_pseudoeffective}, also holds for singular varieties. More precisely, if $D$ is an $\RR$-Cartier
$\RR$-divisor on the normal projective variety $X$, then $\kappa_{\sigma}(D)\geq 0$ if and only if
$D$ is pseudo-effective. Indeed, suppose that $D$ is pseudo-effective and consider an alteration
$\pi\colon Y\to X$, with $Y$ smooth. Since $\pi^*(D)$ is pseudo-effective, we have
$\kappa_{\sigma}(\pi^*(D))\geq 0$ by Proposition~\ref{char_pseudoeffective},
 and since $\kappa_{\sigma}(\pi^*(D))=
\kappa_{\sigma}(D)$ by Proposition~\ref{alteration}, we conclude that $\kappa_{\sigma}(D)\geq 0$.

Conversely, suppose that $\kappa_{\sigma}(D)\geq 0$, and let $(D_m)_{m\geq 1}$ be a sequence of Cartier divisors on $X$ as in Remark~\ref{non_smooth1}. By assumption, there
is an ample Cartier divisor $A$ such that $h^0(X,\cO_X(D_m+A))\geq 1$ for every $m\gg 0$.
Since $\frac{1}{m}(D_m+A)$ is a sequence of pseudo-effective divisors converging to
$D$, it follows that $D$ is pseudo-effective.
\end{remark}

\begin{remark}\label{non_smooth3}
The key ingredient in the proof of Proposition~\ref{char_pseudoeffective}
was the fact that on every smooth projective variety $X$, there is a
Cartier divisor $T$ such that for every big line bundle $L$ on $X$, we have
$h^0(X,L\otimes \cO_X(T))\geq 1$. This, in fact, holds on arbitrary projective varieties.
Indeed, given any projective variety $X$,
consider an alteration $\pi\colon Y\to X$
with $Y$ smooth, and let $T_Y$ be a divisor on $Y$ such that $h^0(Y, L'\otimes\cO_Y(T_Y))\geq 1$ for every big line bundle $L'$ on $Y$. Suppose that $L$ is a big line bundle on $X$. Since
$\pi^*(L)$ is big on $Y$, we have $h^0(Y,\pi^*(L)\otimes T_Y)\geq 1$. Arguing as in the proof of
Proposition~\ref{alteration}, we see that if $H$ is an ample Cartier divisor on $X$, then there are
positive integers $d$ and $d'$ such that $h^0(Y, \cO_Y(d\pi^*(H)-T_Y))\geq 1$
and we have an embedding $\pi_*(\cO_Y)\hookrightarrow \cO_X(d'H)^{\oplus N}$ for some
$N\geq 1$. We conclude that if $T=(d+d')H$, then for every big line bundle $L$ on
$X$, we have $h^0(X,L\otimes \cO_X(T))\geq 1$.

Note that if the Cartier divisor $T$ on $X$ is as above, and if $A$ is any ample Cartier divisor on $X$,
then $G=T+A$ satisfies the conclusion  of Proposition~\ref{char_pseudoeffective}, at least for
Cartier divisors. Indeed, if $D$ is a pseudo-effective Cartier divisor,
then $mD+A$ is big for every $m\geq 1$, hence $h^0(X,\cO_X(mD+G))\geq 1$.
\end{remark}

\begin{remark}\label{Totaro}
As B.~Totaro pointed out, one can alternatively deduce the assertions in
 Proposition~\ref{char_pseudoeffective} and Remark~\ref{non_smooth3} from the results in
 \cite{Arapura} and
\cite{Totaro}. Suppose, for simplicity, that $X$ is a smooth $n$-dimensional projective variety defined over a field of positive characteristic $p$.  It follows from a result of Arapura \cite[Thm. 5.4]{Arapura} that if $H$ is a large
enough multiple of an ample Cartier divisor, then $H$ has the following property: if
$M$ is a line bundle such that $H^n(X, M\otimes\cO_X(-(n+1)H))=0$, then $H^n(X, M^{p^e})=0$ for all $e$ large enough. If $L$ is a big line bundle, then
$H^0(X, \omega_X\otimes L^{p^e})\neq 0$ for $e\gg 0$. It follows from the above property of $H$ and Serre duality that  $H^0(X, \omega_X\otimes\cO_X((n+1)H)\otimes L)\neq 0$.
As we have seen, this is enough to give the statement of
Proposition~\ref{char_pseudoeffective}. One can then deduce the corresponding statement in the singular case as in Remark~\ref{non_smooth3}. Alternatively, one can obtain this directly
using a similar argument to the one above and Totaro's extension \cite[Thm.~5.1]{Totaro}
of Arapura's result to singular varieties.
\end{remark}

\section{A general vanishing statement}

In this section we give an elementary vanishing result, valid in arbitrary characteristic.
This only relies on asymptotic Serre vanishing.

Suppose that $X$ is a normal projective variety over
an algebraically closed field $k$.
Given a nonzero ideal $\fra$ on $X$, we define for every
$\lambda\in\QQ_{\geq 0}$ another ideal $\fra_{\lambda}$, as follows. Let $\pi\colon\widetilde{X}\to X$
be the normalization of the blow-up of $X$ along $\fra$, and let us write
$\fra\cdot\cO_{\widetilde{X}}=\cO_{\widetilde{X}}(-F)$. With this notation, we put
$$\fra_{\lambda}:=\pi_*\cO_{\widetilde{X}}(-\lceil\lambda F\rceil).$$
Note that $\fra_{\lambda}\subseteq\pi_*\cO_{\widetilde{X}}=\cO_X$, hence
$\fra_{\lambda}$ is indeed an ideal of $\cO_X$.% at least if $\lambda \geq 0$ or $\codim V(\mathfrak a) \geq 2$.

\begin{proposition}\label{vanishing1}
Let $\fra$ be a nonzero ideal on $X$ and $B$ a Cartier divisor on $X$ such that
$\fra\otimes\cO_X(B)$ is globally generated. If $E$ is another Cartier divisor on $X$ such that
$E-\lambda B$ is ample for some $\lambda\in\QQ_{\geq 0}$, then there is some $\epsilon>0$ such that for every $\lambda'\in\QQ$ with
$\lambda<\lambda'\leq\lambda+\epsilon$ and every locally free sheaf $\cE$
of finite rank on $X$, we have
$$H^i(X,\cE\otimes\cO_X(\ell E)\otimes\fra_{\ell\lambda'})=0$$
for every $i\geq 1$ and every $\ell\gg 0$ ${\rm (}$depending on $\lambda'$ and $\cE$${\rm )}$.
\end{proposition}

\begin{proof}
Let $\pi\colon\widetilde{X}\to X$ and $F$ be as above, so that $\cO_{\widetilde{X}}(-F)$
is $\pi$-ample. Since $E-\lambda B$ is ample on $X$, it follows that there is
$\epsilon>0$ such that $\pi^*(E-\lambda B)-\eta F$ is ample on $\widetilde{X}$
for $0<\eta\leq\epsilon$.
Note also that by the assumption on $\fra$ and $B$, we can write $\pi^*(B)=F+M$, where
$\cO_{\widetilde{X}}(M)$ is base-point free.

Suppose now that $\lambda'\in\QQ$ satisfies $\lambda<\lambda'\leq\lambda+\epsilon$,
and let $d$ be a positive integer such that $d\lambda$ and $d\lambda'$ are integers. For every $\ell$ we can
write $\ell=qd+r$, for integers $q$ and $r$, with $0\leq r<d$. Therefore $\lceil \ell\lambda' F\rceil
=qd\lambda'F+\lceil r\lambda' F\rceil$. Note that when $\ell$ goes to infinity, then also
$q$ goes to infinity, while there are only finitely many divisors of the form $\lceil r\lambda' F\rceil$.

Since $\cO_{\widetilde{X}}(-F)$ is $\pi$-ample, it follows from the above discussion and relative
asymptotic Serre vanishing that for $\ell\gg 0$, we have
$$R^j\pi_*\cO_{\widetilde{X}}(-\lceil\ell\lambda' F\rceil)=0\,\,\text{for all}\,\,j\geq 1.$$
Whenever this holds, the projection formula and the Leray spectral sequence imply
\begin{equation}\label{eq1_prop1}
H^i(X,\cE\otimes \cO_X(\ell E)\otimes\fra_{\ell\lambda'})\simeq
H^i(\widetilde{X},\pi^*(\cE)\otimes\pi^*\cO_X(\ell E)\otimes\cO_{\widetilde{X}}(-\lceil
\ell\lambda' F\rceil)).
\end{equation}

On the other hand, if $\eta=\lambda'-\lambda$, then we can write
$$\pi^*(\ell E)-\lceil \ell\lambda' F\rceil =\ell\left(\pi^*(E-\lambda B)-\eta F+\lambda M\right) +
\ell\lambda'F-\lceil \ell\lambda' F\rceil.$$
Since $\pi^*(E-\lambda B)-\eta F$ is ample and $M$ is nef, it follows that
$\pi^*(E-\lambda B)-\eta F+\lambda M$ is ample.
If we write as above $\ell=qd+r$, then
$$\ell\left(\pi^*(E-\lambda B)-\eta F+\lambda M\right) +
\ell\lambda'F-\lceil \ell\lambda' F\rceil$$
$$=qd\left(\pi^*(E-\lambda B)-\eta F+\lambda M\right)
+r\left(\pi^*(E-\lambda B)-\eta F+\lambda M\right)+r\lambda'F-\lceil r\lambda'F\rceil.$$
Since when we vary $\ell$ the divisor
 $r\left(\pi^*(E-\lambda B)-\eta F+\lambda M\right)+r\lambda'F-\lceil r\lambda'F\rceil$
 can only take finitely many values, and when $\ell$ goes to infinity, $q$ also goes to infinity,
 it follows from asymptotic Serre vanishing that
 $$H^i(\widetilde{X},\pi^*(\cE)\otimes\pi^*\cO_X(\ell E)\otimes\cO_{\widetilde{X}}(-\lceil
\ell\lambda' F\rceil))=0$$
for all $i\geq 1$ and all $\ell\gg 0$. The assertion in the proposition now follows from this and
the isomorphism (\ref{eq1_prop1}).
\end{proof}

\begin{remark}
For our purpose the precise definition of $\fra_{\lambda}$ will not be  important.
For example, we might have taken instead $\fra_{\lambda}:=
\pi_*\cO_{\widetilde{X}}(-\lfloor \lambda
F\rfloor)$, and in characteristic zero we might have taken $\fra_{\lambda}=\cJ(\fra^{\lambda})$,
the multiplier ideal of $\fra$ of exponent $\lambda$.
The proof of Proposition~\ref{vanishing1} works also with these definitions,
and the first variant would work as well for the applications in the next section.
The definition that we have considered gives the
\emph{integrally closed rational powers} of $\fra$, see \cite[\S 10.5]{HS}.
\end{remark}

\begin{corollary}\label{cor_vanishing}
The assertion in Proposition~\ref{vanishing1} also holds if instead of assuming $\cE$
locally free we only assume that ${\mathcal Tor}_i^{\cO_X}(\cE,\fra_{\ell\lambda'})=0$
for all $i\geq 1$, all $\lambda'\in\QQ_{>0}$, and all $\ell\gg 0$
${\rm (}$depending on $\lambda'$ and $\cE$${\rm )}$.
\end{corollary}

\begin{proof}
Since $X$ is projective, we can find a (possibly infinite) resolution
$$\cdots\to\cE_i\to\ldots\to \cE_1\to\cE_0\to\cE\to 0,$$
with all $\cE_i$ locally free $\cO_X$-modules of finite rank.
Our assumption on $\cE$ implies that after tensoring this complex by
$\cO_X(\ell E)\otimes \fra_{\ell\lambda'}$, the resulting complex is still exact.
By chasing the resulting short exact sequences, we see that if
$\lambda'>\lambda$ is close enough to $\lambda$, and
$\ell$ is large enough (depending on $\lambda'$) so that the conclusion of Proposition~\ref{vanishing1} is satisfied for all $\cE_i$ with $0\leq i\leq n-1$, where
$n=\dim(X)$, then we obtain the conclusion of the corollary.
\end{proof}

\section{Proofs of the main results}

Our first goal is to prove Theorem~\ref{p_lifting} from the Introduction.
We henceforth assume that the ground field $k$ is algebraically closed, of
characteristic $p>0$.
A key ingredient
 is the log trace map with respect to the Frobenius morphism,
that we now review.

Suppose that $X$ is a smooth variety over $k$. Let $F=F_X\colon X\to X$ denote the absolute Frobenius morphism on $X$, which is the identity on the topological space and which takes
a regular function $u$ to $u^p$.
We have a canonical surjective map
$$t_{X}\colon F_*\omega_X\to \omega_X,$$
which can be
either defined as the trace map for duality with respect to $F$, or as coming from the Cartier isomorphism. Given algebraic coordinates $x_1,\ldots,x_n$ on an open subset $U$ of $X$, this map is characterized by
$$t_X(x_1^{i_1}\cdots x_n^{i_n}dx_1\wedge\cdots\wedge dx_n)=
x_1^{\frac{i_1-p+1}{p}}\cdots x_n^{\frac{i_n-p+1}{p}}dx_1\wedge\cdots\wedge dx_n,$$
where the monomial on the right-hand side is understood to be zero if one of the exponents is not an integer.
Iterating this map $e$ times we obtain a surjective map $t_X^e\colon F^e_*(\omega_X)\to\omega_X$.

Suppose now that $E=E_1+\ldots+E_N$ is a simple normal crossing divisor on $X$.
By tensoring $t_X$ with $\cO_X(E)$, and composing with the inclusion
$$F_*(\omega_X(E))\hookrightarrow F_*(\omega_X(pE))\simeq
 (F_*(\omega_X))\otimes\cO_X(E),$$ we obtain a \emph{log trace map}
 $$t_{X,E}\colon F_*(\omega_X(E))\to\omega_X(E).$$
 Using the above explicit description in terms of a system of coordinates on $U$ with the property
 that each prime divisor in $E$ which intersects $U$ is defined by some $x_i=0$, it is easy to see
 that $t_{X,E}$ is again surjective. After iterating  $e$ times $t_{X,E}$, we obtain
 $$t_{X,E}^e\colon F^e_*(\omega_X(E))\to\omega_X(E).$$

Note that this construction is compatible with adjunction. More precisely, given
$(X,E)$ as above, let $Y=E_1\cap\ldots\cap E_r$. This is a smooth variety of
codimension $r$ in $X$ (possibly
not connected), and $E$ induces a simple normal crossing divisor
$E_Y=\sum_{i=r+1}^NE_i\vert_Y$. Adjunction gives an isomorphism
$\omega_X(E_1+\ldots+E_r)\vert_Y\simeq\omega_Y$, and it follows from the above description
of the log trace maps in local coordinates that the diagram

\begin{equation}\label{comm_diag1}
\begin{CD}
(F_X)_*(\omega_X(E))@>>>
(F_Y)_*(\omega_Y(E_Y))\\
@V{t_{X,E}}VV @VV{t_{Y,E_Y}}V\\
\omega_X(E)@>>> \omega_Y(E_Y),
\end{CD}
\end{equation}
is commutative, where the horizontal maps are induced by restriction to $Y$
and the adjunction isomorphism.

%We can now prove the main result of this paper.

%\todoK{What do people think of putting the statement of the theorem here too, and also 1.1 before the proof of 1.1, just for ease of reference for notation.}
\noindent{\bf Theorem \ref{p_lifting}. }{\it Let $X$ be a smooth projective variety over an
algebraically closed field $k$ of characteristic $p>0$, and let $D$ be an $\RR$-divisor on $X$. Suppose that $H=H_1+\ldots+H_r$ is a simple normal crossing divisor on
$X$, with $r<\dim(X)$,
 and that
$W:=H_1\cap\dots\cap H_r$  does not intersect the non-nef locus $\mathbf B_-(D)$ of $D$.
In this case there exists an ample divisor $G$ on $X$ such that
if $$A=K_X+H+2G$$
then $A$ is ample and
the restriction map
\begin{equation}\label{e_lifting}
H^0(X,\cO_X(\lfloor mD\rfloor +A))\to H^0(W,\cO_X(\lfloor mD\rfloor +A)\vert_W)
\end{equation}
is surjective for every $m\ge 1$.
}
\begin{proof}%[Proof of Theorem~\ref{p_lifting}]
Let $\mathcal F$ be the kernel of the trace morphism
$$t_W\colon (F_W)_*(\omega_W)\to \omega_W.$$
By Fujita's vanishing theorem (see \cite{Fujita}), if $G$ is a large enough multiple of
a given ample divisor on $X$, then
\begin{equation}
\label{eq.H1Vanishing}
H^1(W,\mathcal F\otimes \cO_X(G)\vert_W\otimes L)=0
\end{equation}
for every nef line bundle $L$ on $W$.

Let us write $D=\alpha_1D_1+\ldots+\alpha_sD_s$, with the $D_i$ distinct prime divisors on $X$ and $\alpha_i\in \RR$. After possibly replacing $G$ by a multiple, we may assume that
$G-(t_1D_1+\ldots+t_sD_s)$ is ample for all $t_1,\ldots,t_s\in [0,1]$.
In this case $G-\{mD\}$ is ample for every $m\geq 1$, where
$\{mD\}=mD-\lfloor mD\rfloor$. Since
$\lfloor mD\rfloor+G=mD+(G-\{mD\})$, it follows from our assumption that
the stable base locus ${\rm SB}(\lfloor mD\rfloor+G)$ does not intersect $W$.
Let $r_m\geq 1$ be such that the base locus of $|r_m(\lfloor mD\rfloor+G)|$ does not intersect $W$, and let us denote by $\fra^{(m)}$ the ideal defining the base-locus of this linear system,
with its natural scheme structure.

Let $A=K_X+H+2G$, where $K_X$ is such that $\omega_X=\cO_X(K_X)$. After possibly replacing $G$ by a multiple, we may assume that $A$ is ample.  Let us fix $m\geq 1$. In order to prove the surjectivity of (\ref{e_lifting}),
we first apply Corollary~\ref{cor_vanishing}
for the ideal $\fra^{(m)}$, $\lambda_m=\frac{1}{r_m}$, the divisors
$B_m=r_m(\lfloor mD\rfloor+G)$ and
$E_m=\lfloor mD\rfloor +2G$, and the sheaf $\cE_m=\cI_W\otimes\omega_X(H)$,
where $\cI_W$ is the ideal defining the subvariety $W$.
Note that since $W$ is disjoint from the zero-locus of $\fra^{(m)}$,
we have ${\mathcal Tor}_i^{\cO_X}(\cE_m, \fra^{(m)}_{\mu})=0$
for every $\mu\in\QQ_{>0}$ and every $i\geq 1$. The vanishing given by Corollary~\ref{cor_vanishing}
implies that we can find $\lambda'_m>\lambda_m$ such that
for $e\gg 0$, the restriction map
\begin{equation}\label{eq1_theorem}
H^0(X, \omega_X(H)\otimes\cO_X(p^eE_m)\otimes\fra^{(m)}_{p^e\lambda'_m})
\to H^0(W, \omega_W\otimes\cO_X(p^eE_m)\vert_W)
\end{equation}
is surjective (note that $\fra^{(m)}_{p^e\lambda'_m}\cdot\cO_W=\cO_W$,
since $W$ does not intersect the zero-locus of
$\fra^{(m)}$). In particular, the restriction map
\begin{equation}\label{eq2_theorem}
H^0(X, \omega_X(H)\otimes\cO_X(p^eE_m))
\to H^0(W, \omega_W\otimes\cO_X(p^eE_m)\vert_W)
\end{equation}
is surjective for $e\gg 0$.

It follows from (\ref{comm_diag1}) after iteration that we also have a commutative diagram
\begin{equation}\label{diag1}
\begin{CD}
(F_X^e)_*(\omega_X(H)) @>>> (F_W^e)_*(\omega_W) \\
@V{t_{X,H}^e}VV@VV{t_W^e}V \\
\omega_X(H)@>>>  \omega_W
\end{CD}
\end{equation}
in which the horizontal maps are induced by restriction via adjunction,
and the vertical maps are the corresponding iterated trace maps.
Tensoring with $\cO_X(E_m)$ and using the projection formula, we obtain the commutative diagram
\begin{equation}\label{diag2}
\begin{CD}
(F_X^e)_*(\omega_X(H)\otimes\cO_X(p^eE_m)) @>>> (F_W^e)_*(\omega_W
\otimes\cO_X(p^eE_m)\vert_W) \\
@VVV@VVV \\
\omega_X(H)\otimes\cO_X(E_m)@>>>  \omega_W\otimes\cO_X(E_m)\vert_W.
\end{CD}
\end{equation}

By taking global sections, we obtain the commutative diagram
\begin{equation}
\begin{CD}\label{diag3}
H^0\big(X, (F_X^e)_*(\omega_X(H)\otimes\cO_X(p^eE_m))\big) @>>> H^0\big(W,(F_W^e)_*( \omega_W
\otimes\cO_X(p^eE_m)\vert_W)\big) \\
@VVV@VVV \\
H^0(X, \omega_X(H)\otimes\cO_X(E_m))@>>> H^0(W,  \omega_W\otimes\cO_X(E_m)\vert_W),
\end{CD}
\end{equation}
in which the top horizontal map is surjective for $e\gg 0$ by (\ref{eq2_theorem}).

We claim that the right vertical map is surjective for every $e\geq 1$. In order to prove this, it is enough to show that
$$H^0\big(W, (F^{i+1}_W)_*(\omega_W\otimes\cO_X(p^{i+1}E_m)\vert_W)\big)
\to H^0\big(W, (F^i_W)_*(\omega_W\otimes\cO_X(p^iE_m)\vert_W)\big)$$
is surjective for every $i\geq 0$. It follows from the exact sequence
$$0\to (F^i_W)_*(\cF)\to (F^{i+1}_W)_*(\omega_W)\to (F^i_W)_*(\omega_W)\to 0$$
that it is enough to show that $H^1(W, (F^i_W)_*(\cF\otimes\cO_X(p^iE_m)\vert_W))=0$,
or equivalently,
\begin{equation}\label{eq_vanishing}
H^1(W, \cF\otimes\cO_X(p^iE_m)\vert_W)=0.
\end{equation}
Recall that $\cO_X(\lfloor mD\rfloor +G)\vert_W$ is nef (in fact, semiample).
Therefore
$$\cF\otimes \cO_X(p^iE_m)\vert_W\simeq\cF\otimes \cO_X(G)\vert_W
\otimes\cO_X(p^i(\lfloor mD\rfloor+G)+
(p^i-1)G)\vert_W,$$
hence the vanishing in (\ref{eq_vanishing}) follows from \eqref{eq.H1Vanishing}.

Since both the right vertical and top horizontal maps in
(\ref{diag3}) are surjective for $e\gg 0$, it follows that also the bottom horizontal map in that diagram is surjective, which is precisely what we needed to prove.
\end{proof}

\begin{remark}
It follows from the above proof that the surjectivity in the statement of Theorem~\ref{p_lifting}
also holds if we replace $A$ by any divisor $A'$ such that $A'-A$ is ample.
\end{remark}

\begin{remark}\label{curve_case}
If the subvariety $W$ in Theorem~\ref{p_lifting} is a curve, then we can be more explicit
about the choice of $G$ such that we have the vanishing in (\ref{eq.H1Vanishing}). Indeed,
 by a theorem of Tango (cf. \cite{Tan72}) the vanishing in (\ref{eq.H1Vanishing})
 holds if ${\rm deg}(G\vert_W)>\frac{2g-2}{p}$, where $g$ is the genus of $W$ and
 ${\rm char}(k)=p$.
\end{remark}

\begin{remark}
Instead of restricting to $W$ in one step via adjunction in diagram \eqref{diag1}, it is possible to cut down by each $H_i$ individually, and keep track of the sections that extend via vector subspaces similar to the $S^0$ defined in \cite{Schwede}.  In fact, it follows from \eqref{diag3} that a subspace of $S^0(X,\cO_X(\lfloor mD\rfloor +A))$ surjects onto $H^0(W,\cO_X(\lfloor mD\rfloor +A)\vert_W)$ for all $m \geq 1$.
\end{remark}

\begin{remark}
Finally, we remark that it is possible to weaken the hypothesis that $X$ is smooth and $H$ is simple normal crossing to simply that: there exists an open neighborhood $U$ containing $W := H_1 \cap \ldots \cap H_r$ such that $U$ is smooth and $H|_U$ is simple normal crossing.  The proof is unchanged.  In fact, it is possible even to obtain similar statements if $W$ is an $F$-pure center of $(X, H)$.
\end{remark}

We use the method in the proof of Theorem~\ref{p_lifting} to also prove the lower bound on numerical dimension stated in the Introduction.

\noindent{\bf{Theorem \ref{thm_main}.}} {\it Suppose that $X$ is a smooth projective variety over an algebraically closed field $k$ of positive characteristic.  If $D$ is a pseudo-effective
$\RR$-divisor
on $X$ which is not numerically equivalent to the negative part $N_{\sigma}(D)$ in its divisorial Zariski
decomposition, then $\kappa_{\sigma}(D)\geq 1$, that is, there is an ample divisor
$A$ on $X$ and $C>0$ such that
$$h^0(X,\cO_X(\lfloor mD\rfloor+A))\geq Cm\,\,\text{for all}\,\,m\gg 0.$$
}
%\vskip 1pt
\begin{proof}%[Proof of Theorem~\ref{thm_main}]
Let $D=N_{\sigma}(D)+P_{\sigma}(D)$ be the divisorial Zariski decomposition of $D$.
We simply write $N_{\sigma}$ and $P_{\sigma}$ for $N_{\sigma}(D)$ and $P_{\sigma}(D)$,
respectively.
We fix a very ample divisor $H$ on $X$.
By assumption, $P_{\sigma}$ is not numerically trivial, and it is pseudo-effective by
Proposition~\ref{prop_Nakayama}. It follows from Lemma~\ref{lem_num_trivial1}
that $(P_{\sigma}\cdot H^{n-1})>0$, where $n=\dim(X)$.

For each $m\geq 1$, we will consider a curve $W_m$ in $X$,
given as the intersection of general
$(n-1)$ elements in the linear system $|H|$. Note that each such $W_m$ is smooth and connected, of genus
$g=\frac 1 2 ((K_X+(n-1)H)\cdot H^{n-1})+1$. We fix an ample divisor $G$ on $X$ such that
$(G\cdot H^{n-1})>\frac{2g-2}{p}$, where $p={\rm char}(k)$. It follows from
Remark~\ref{curve_case} that for every $W_m$ as above, the vanishing in
(\ref{eq.H1Vanishing}) holds.

Furthermore, arguing as in the proof of Theorem~\ref{p_lifting}, we see that after possibly
replacing $G$ by a multiple, we may assume that $G-\{mP_{\sigma}\}$ is ample for every
$m\geq 1$. Since $\lfloor mP_{\sigma}\rfloor+G=mP_{\sigma}+(G-\{mP_{\sigma}\})$,
it follows from Proposition~\ref{prop_Nakayama} that
$\SB(\lfloor mP_{\sigma}\rfloor+G)$ contains no codimension one subvarieties.
We can therefore choose $W_m$ as above for each $m\geq 1$ such that
$W_m\cap \SB(\lfloor mP_{\sigma}\rfloor+G)=\emptyset$.

It follows from the proof of Theorem~\ref{p_lifting} that if we take $A=K_X+H+2G$
(which may be assumed ample after replacing $G$ by a multiple), then
$$h^0(X,\cO_X(\lfloor mP_{\sigma}\rfloor+A))\geq h^0(W_m,\cO_X(\lfloor mP_{\sigma}\rfloor+A)\vert_{W_m})$$
for all $m\geq 1$.
Since $W_m$ is a smooth curve of genus $g$ and $(P_{\sigma}\cdot H^{n-1})>0$, we deduce from the Riemann-Roch theorem that
$$h^0(X,\cO_X(\lfloor mP_{\sigma}\rfloor+A))\geq ((\lfloor mP_{\sigma}\rfloor+A)\cdot H^{n-1})-g
\geq Cm$$
for a suitable $C>0$ and all $m\gg 0$.
Since $N_{\sigma}$ is effective, the difference $\lceil mD\rceil-\lfloor mP_{\sigma}\rfloor$
is effective, hence
$$h^0(X,\cO_X(\lceil mD\rceil+A))\geq h^0(X,\cO_X(\lfloor mP_{\sigma}\rfloor+A))\geq Cm$$
for $m\gg 0$. As we have seen in the discussion of the definition of numerical dimension in
\S 2, this implies $\kappa_{\sigma}(D)\geq 1$.
\end{proof}

We conclude with an application of Theorem~\ref{p_lifting}, showing that also in positive
characteristic, the definition of numerical dimension that we have been using agrees in the case
of a nef $\RR$-divisor with the definition in \cite{Kawamata}. Suppose that $X$ is a smooth $n$-dimensional projective variety over an algebraically closed field of positive characteristic. If $D$ is a nef
$\RR$-divisor on $X$, let us temporarily denote by $\nu(D)$ the largest $j\geq 0$ such that the
cycle class $D^j$ is not numerically trivial. It follows from Lemma~\ref{lem_num_trivial2}
that if $H$ is an ample divisor on $X$, then $\nu(D)$ is the largest $j$ such that
$(D^j\cdot H^{n-j})\neq 0$.

\begin{proposition}
\label{prop.KawamataNumericalDimensionEquivalent}
If $D$ is a nef $\RR$-divisor on the smooth projective variety $X$ as above, then
$\kappa_{\sigma}(D)=\nu(D)$.
\end{proposition}

%\todoK{What do people think of including this in the introduction?}

\begin{proof}
We first show that $\kappa_{\sigma}(D)\geq\nu(D)$. Let $n=\dim(X)$. If $\nu(D)=n$, then
$D$ is big, and in this case it is clear that $\kappa_{\sigma}(D)\geq n$. Suppose now that
$\nu(D)=j<n$, and let $H$ be a very ample divisor on $X$. If  $H_1,\ldots,H_{n-j}$
are general elements in the linear system $|H|$, then $H_1+\ldots+H_{n-j}$ has simple normal crossings.
It follows from Theorem~\ref{p_lifting} that if $W=H_1\cap\ldots\cap H_{n-j}$, then there is an ample divisor $A$ on $X$ such that
$$H^0(X,\cO(\lfloor mD\rfloor+A))\to H^0(W,\cO_X(\lfloor mD\rfloor+A)\vert_W)$$
is surjective for every $m\geq 1$. On the other hand, since $(D^j\cdot H^{n-j})>0$, it follows that
$\cO_X(D)\vert_W$ is big.
Furthermore, after possibly replacing $H$ by a multiple, we may assume that
the following holds: if we write $D=
a_1D_1+\ldots+a_sD_s$, with the $D_i$ distinct prime divisors and $a_i\in \RR$, then also the
$D_i\vert_W$ are distinct
prime divisors. In particular, $D\vert_W$ is well-defined as a divisor and
$\lfloor mD\rfloor\vert_W=\lfloor mD\vert_W\rfloor$.
By putting all these together, we conclude that
 there is $C>0$ such that
$$h^0(X,\cO(\lfloor mD\rfloor+A))\geq h^0(W,\cO_X(\lfloor mD\rfloor+A)\vert_W)
\geq Cm^j$$
for all $m\gg 0$, hence $\kappa_{\sigma}(D)\geq j$.

We prove the reverse inequality by induction on $n$. If $D$ is big, then $\nu(D)=n$,
hence we are done. Otherwise, let $r=\kappa_{\sigma}(D)$, and let $A$ be an ample
divisor on $X$ such that
$$\liminf_{m\to\infty} \frac {h^0(X,\cO_X(\lfloor mD\rfloor+A))}{m^r}>0.$$
If $n=1$, then $r=0$ and there is nothing to prove. Therefore we assume $n\geq 2$.

Let $\ell$ be a positive integer such that $(\ell+1) A$ is very ample, and choose $E$ a general element in the linear system $|(\ell +1)A|$. We have an exact sequence
$$0\to H^0(X,\cO_X(\lfloor mD\rfloor-\ell A))\to H^0(X,\cO_X(\lfloor mD\rfloor+A))
\to H^0(E, \cO_X(\lfloor mD\rfloor+A)\vert_E).$$
Since $D$ is not big, it follows that $H^0(X,\cO_X(\lfloor mD\rfloor-\ell A))=0$ for every $m\geq 1$, and the exact sequence implies
\begin{equation}\label{eq_final}
\liminf_{m\to\infty} \frac {h^0(E,\cO_X(\lfloor mD\rfloor+A)\vert_E)}{m^r}>0.
\end{equation}
Since $E$ is general, we may assume that $E$ is smooth. Furthermore,
arguing as above, we see that we may assume that
$D\vert_E$ is well-defined as a divisor, and
$\lfloor mD\rfloor\vert_E=\lfloor mD\vert_E\rfloor$ for every $m$.
Therefore
(\ref{eq_final}) gives
$\kappa_{\sigma}(D)\leq\kappa_{\sigma}(D\vert_E)$. On the other hand, since $D$ is not big
and $E$ is ample, we have $\nu(D)=\nu(D\vert_E)$, and using the inductive assumption we obtain
$$\kappa_{\sigma}(D)\leq \kappa_{\sigma}(D\vert_E)\leq \nu(D\vert_E)=\nu(D).$$
This completes the proof of the proposition.
\end{proof}

\begin{remark}
In fact, the assertion in Proposition~\ref{prop.KawamataNumericalDimensionEquivalent}
also holds on singular projective varieties. More precisely, suppose that $X$ is a
normal projective variety
over an algebraically closed field of positive characteristic, and let $n=\dim(X)$. If $D$ is a nef
$\RR$-Cartier $\RR$-divisor on $X$, then $\kappa_{\sigma}(D)=\nu(D)$, where if $H$
is an ample Cartier divisor on $X$, we denote by $\nu(D)$ the largest $j$ such that
$(D^j\cdot H^{n-j})\neq 0$. Indeed, let $\pi\colon Y\to X$ be an alteration, with $Y$ smooth.
We have $\kappa_{\sigma}(D)=\kappa_{\sigma}(\pi^*(D))$
by Proposition~\ref{alteration} and $\kappa_{\sigma}(\pi^*(D))=\nu(\pi^*(D))$ by
Proposition~\ref{prop.KawamataNumericalDimensionEquivalent}. Therefore it is enough to show that $\nu(D)=\nu(\pi^*(D))$.

Let $r=\nu(D)$.
Note first that $(\pi^*(D)^{r}\cdot \pi^*(H)^{n-r})={\rm deg}(\pi)(D^r\cdot H^{n-r})\neq 0$, hence
$\pi^*(D)^{r}$ is not numerically trivial. We similarly obtain
$(\pi^*(D)^{r+1}\cdot \pi^*(H)^{n-r-1})=0$. Note that $\pi^*(H)$ is nef and big. In particular, we can write $\pi^*(H)=A+E$, for $\RR$-divisors $A$ and $E$ on $X$, with $A$ ample and $E$
effective. Since $\pi^*(H)$ and $\pi^*(D)$ are nef, we have
$$(\pi^*(D)^{r+1}\cdot A^{n-r-1})\leq (\pi^*(D)^{r+1}\cdot \pi^*(H)^{n-r-1})=0.$$
It follows from Proposition~\ref{lem_num_trivial2} that $\nu(\pi^*(D))=r$, which completes the proof of our assertion.
\end{remark}

\begin{remark}
J.~Koll\'{a}r pointed out to us that one can give a proof of
Proposition~\ref{prop.KawamataNumericalDimensionEquivalent} using Fujita's vanishing theorem and Matsusaka's results on variable intersection cycles
\cite{Matsusaka}. His argument, in fact, works directly on arbitrary normal varieties.
\end{remark}

\providecommand{\bysame}{\leavevmode \hbox \o3em
{\hrulefill}\thinspace}

\end{document}